\documentclass[11pt]{article}

\usepackage[a4paper,margin=1in]{geometry}
\usepackage{amsmath,amsthm,amsfonts,amssymb}
\usepackage{mathtools}
\usepackage{tikz}
\usepackage{graphicx}
\usepackage{booktabs}
\usepackage{caption}
\usepackage{placeins}
\usepackage{hyperref}
\hypersetup{hidelinks}
\usepackage{url}
\usepackage[mathlines]{lineno}

\newcommand{\NN}{\mathbb{N}}
\newcommand{\ZZ}{\mathbb{Z}}
\newcommand{\QQ}{\mathbb{Q}}
\newcommand{\eps}{\varepsilon}

\theoremstyle{plain}
\newtheorem{theorem}{Theorem}

\newtheorem{proposition}{Proposition}
\newtheorem{corollary}{Corollary}
\newtheorem{conjecture}{Conjecture}

\theoremstyle{definition}
\newtheorem{definition}{Definition}

\theoremstyle{remark}
\newtheorem{remark}{Remark}

\title{\bf No-three-in-line sets on the checkerboard grid}
\author{Thomas Prellberg\smallskip\\
\small School of Mathematical Sciences\\
\small Queen Mary University of London\\ 
\small United Kingdom\smallskip\\
\small \texttt{t.prellberg@qmul.ac.uk}}
\date{\today}

\begin{document}
\maketitle

\begin{abstract}
The classical no-three-in-line problem asks for the largest number $D(n)$ of points that can be
chosen from an $n\times n$ grid with no three collinear.  It remains open whether the elementary
upper bound $2n$ is always attainable.  We study a checkerboard-restricted variant in which all
chosen points must be monochromatic under the colouring of the grid by the parity of $x+y$.  If
$D_{\mathrm{mono}}(n)$ denotes the largest number of monochromatic points with no three collinear,
then the monochromatic diagonals already give the elementary bound $D_{\mathrm{mono}}(n)\le 2n-2$.

The main object of the paper is a four-direction linear-programming relaxation on a fixed colour
class, using rows, columns, and the two diagonal families of slopes $\pm1$.  For the ordinary
square-grid problem this relaxation gives the trivial bound.  On the checkerboard it gives
substantially tighter finite bounds in the computed range.  For $2\le n\le 16$ its floor agrees with
the exact single-colour optimum except at four side lengths, where the gap is one.  After symmetry
reduction, the dual relaxation has three one-dimensional reduced forms, according to the parity of
$n$ and the chosen colour class; their computed optimizers display structured piecewise-profile
patterns.

The central construction is an exact continuum dual certificate for the formal continuum problem
associated with the scaled symmetry-reduced odd-fat case.  We construct explicit nonnegative
functions $A$ and $B$ satisfying the continuum obstacle inequalities and having objective value
\(\alpha\), where
\[
401\alpha^3-1744\alpha^2+2240\alpha-768=0
\]
and \(\alpha\) is the middle real root, approximately $1.5768$.
Combined with a continuum-to-discrete sampling theorem of Aujla, Prellberg and Sandhu, this
certificate yields the asymptotic upper bound
\[
D_{\mathrm{mono}}(n)\le \alpha n+O(1)
\]
for the checkerboard problem.  Finite LP computations are consistent with $\alpha$ as the exact
limiting slope of the relaxation, and the exact small-$n$ data suggest, more speculatively, that the
true checkerboard optimum may track the same scale.  We state the corresponding matching lower
bounds as conjectures.
\end{abstract}

\section{Introduction}

The no-three-in-line problem, introduced by Dudeney in recreational form and later developed in the
combinatorics literature, asks for the maximum number $D(n)$ of points that can be selected from the
$n\times n$ integer grid so that no three lie on a common line
\cite{dudeney1917,guy1968,hall1975}.  The pigeonhole principle gives the elementary upper bound
$D(n)\le 2n$.  The problem of attaining this bound has been studied extensively.  See, for
example,
\cite{adena1974,anderson1979,flammenkamp1992,flammenkamp1998,harborth1989,kloeve1979,prellberg2026}
and the broader accounts in \cite{bmp2005,eppstein2018}.

Several neighbouring variants are close enough to be worth separating from the present one.  Finite-torus
and other modular versions ask for collinearity modulo a finite abelian group
\cite{fowler2012,kuwong2018,skotnica2019}; permutation or transversal versions study related
collinearity restrictions in permutation-like sets \cite{cooperhyatt2025,coopersolymosi2005}; and
higher-dimensional grid versions have also been studied \cite{porwood2007}.  Other nearby problems
concern minimum or extensible no-three-in-line sets
\cite{aichholzer2023,cooperpikhurko2014,nagy2023}, as well as no-$(k+1)$-in-line generalisations
\cite{kovacs2025}.  For arbitrary input point sets, the general-position subset selection problem is
studied both combinatorially and algorithmically
\cite{balogh2024,dumitrescu2025,froese2017,furedi1991,paynewood2013}.

This paper studies a parity-restricted version of the same problem.  Colour the points $(x,y)$ of
the $n\times n$ grid by the parity of $x+y$, as on a checkerboard.  Under this colouring,
$D_{\mathrm{mono}}(n)$ denotes the largest number of monochromatic points with no three collinear,
and $D_{\mathrm{mono}}(n,\eps)$ denotes the maximum within the colour class $\eps\in\{0,1\}$.
Formal definitions are deferred to Section~\ref{sec:notation}.
When $n$ is odd, the two colour classes differ in size by one; we call the larger one the
\emph{fat} class and the smaller one the \emph{thin} class.  For even $n$ the two classes have
equal size.  The adjective ``monochromatic'' is used here only with respect to this fixed
checkerboard colouring: a feasible set is required to lie in one of the two colour classes.  This is
different from colouring the whole grid so that every colour class has the no-three-in-line
property, as in Wood's grid-colouring problem \cite{wood2004}, and from graph-theoretic
position-set colourings \cite{chandran2025}.  The author has not found a prior published treatment
of the present square-grid parity-class maximisation problem.  The neighbouring colouring and
general-position references above provide context, but no result from them is used in the proofs
below.

Since every monochromatic no-three-in-line set is in particular a no-three-in-line set, one has
\[
D_{\mathrm{mono}}(n)\le D(n)\le 2n .
\]
The checkerboard restriction is motivated by a simple obstruction to extending maximal
configurations.  Suppose that a $2n$-point no-three-in-line set on an $n\times n$ grid is dilated by
a factor of two and placed in the $2n\times2n$ grid.  Since a $2n$-point configuration uses exactly
two points in every row and every column of the original grid, the dilated configuration already
saturates all even rows and all even columns of the larger grid.  Any additional points in a putative
completion to a $4n$-point configuration would therefore have to lie in odd rows and odd columns.
Thus all added points, together with the dilated points, would lie in a single checkerboard colour
class.  As a consequence of the capacity bound below, such a completion is impossible
(Corollary~\ref{cor:no-dilation}).  Single-colour no-three-in-line sets are therefore one natural
finite-grid test case for obstructions to extension arguments of this kind.

There is also an immediate capacity obstruction.  In one checkerboard class, every diagonal of
slope $+1$ and every diagonal of slope $-1$ is monochromatic.  Each such diagonal can therefore
contain at most two selected points, and counting the capacity of one diagonal family gives
\[
D_{\mathrm{mono}}(n)\le 2n-2.
\]
Section~\ref{sec:capacity} proves this bound.  The first arXiv version of this paper suggested,
by a boundary-forcing heuristic, the sharper inequality $D_{\mathrm{mono}}(n)\le 2n-4$ for
$n\ge6$.  This has since been proved, by a moment argument, in the companion paper
\cite{aujla2026}; see Section~\ref{sec:capacity}.

The main tool in the paper is the four-direction linear-programming relaxation
$L_{\mathrm{mono}}(n,\eps)$.  This relaxation replaces the integral choice of points in $C_\eps$ by a
nonnegative fractional mass assignment and keeps the capacity-two constraints on rows, columns,
and the two diagonal families of slopes $\pm1$.  Hence
\[
D_{\mathrm{mono}}(n,\eps)\le L_{\mathrm{mono}}(n,\eps).
\]
For the ordinary square-grid no-three-in-line problem, the analogous relaxation gives the trivial
bound $2n$.  On the checkerboard, however, the same four directions become much more
informative in the computations reported here.  Separate exact search for $2\le n\le16$ shows that
the floor of the LP optimum agrees with $D_{\mathrm{mono}}(n,\eps)$ except in four side lengths, where the
gap is one.

The dual form of this LP is particularly useful.  A dual solution assigns nonnegative weights to
rows, columns, and diagonals, with the requirement that every point of the chosen colour class is
covered by total weight at least one.  Any such weighting is an explicit upper-bound certificate.  By
averaging over the symmetries preserving the chosen colour class, the dual problem reduces to
one-dimensional line-weight profiles.  The three cases are odd side length with the fat class, odd
side length with the thin class, and even side length.  The computed reduced LP optimizers display a
stable piecewise-profile pattern and point numerically to a common linear scale
\[
L_{\mathrm{mono}}(n,\eps)\approx \alpha n,
\qquad
\alpha\approx 1.5768.
\]
At the small sizes where exact search is feasible, the true values $D_{\mathrm{mono}}(n,\eps)$ are
consistent with the same scale.

The main construction of the paper is an exact continuum upper-bound certificate at this constant
in the cleanest of the three symmetry types, namely the odd-fat case.  A formal scaling of the
reduced variables leads to a continuum obstacle problem.  One seeks nonnegative functions
$A,B$ on $[0,1]$ minimizing
\[
4\int_0^1 (A(t)+B(t))\,dt
\]
subject to
\[
A(x)+A(1-y)+B(x+y)+B(x-y)\ge 1
\]
on the triangle
\[
T=\{(x,y):0\le y\le x,\ x+y\le1\}.
\]
Theorem~\ref{thm:odd-fat-continuum-certificate}, proved in
Subsection~\ref{subsec:odd-fat-continuum}, constructs explicit piecewise quadratic and linear
functions $A$ and $B$ satisfying these inequalities.  Their objective value is the middle real root
$\alpha$ (approximately $1.5768$) of
\[
401\alpha^3-1744\alpha^2+2240\alpha-768=0.
\]
The obstacle inequality is verified by an exact algebraic Bernstein-basis certificate over a finite
subdivision of the continuum domain.

The construction proves the upper bound $\Lambda_{\mathrm{fat}}\le\alpha$ for the odd-fat
continuum relaxation.  To turn this into a statement about the finite grids one needs a
continuum-to-discrete theorem.  Such a theorem is proved in the companion paper
\cite{aujla2026}: sampling a feasible continuum pair $(A,B)$ of bounded variation at the mesh
points produces feasible reduced dual weights for every finite grid, at a cost of $O(1)$ in the
objective.  Applied to the certificate of Theorem~\ref{thm:odd-fat-continuum-certificate}, it gives
the main result of the present paper,
\[
D_{\mathrm{mono}}(n)\le L_{\mathrm{mono}}(n,\eps)\le \alpha n+O(1),
\]
stated as Theorem~\ref{thm:discrete-bound} in Subsection~\ref{subsec:transfer}.  The sampling
argument, which was also suggested independently by a referee of this paper, is short in the
odd-fat case and is sketched there.

The bound is one-sided.  The finite LP data support the conjecture that $L_{\mathrm{mono}}(n,\eps)/n$
tends to $\alpha$ in all three symmetry types, and the exact small-$n$ data suggest the same scale
for the monochromatic no-three-in-line problem itself.  We formulate the corresponding lower-bound
statements as conjectures following the reduced LP evidence.

The paper is organised as follows.  Section~\ref{sec:notation} fixes the basic notation and parity
terminology.  Section~\ref{sec:capacity} gives the elementary diagonal bound.  Section~\ref{sec:lp} introduces
the four-direction LP relaxation and its dual certificate interpretation.  Proposition~\ref{prop:symmetry-reduced-duals} records the three symmetry-reduced dual
programs in Subsection~\ref{subsec:lp-reduced-cases}.  Subsection~\ref{subsec:odd-fat-continuum} states and verifies the odd-fat continuum certificate, and
Subsection~\ref{subsec:transfer} records its transfer to the finite grids.
Appendix~\ref{app:profile-curvature}
records finite profile-curvature data, and Appendix~\ref{app:certificate-derivation} gives the active-patch ansatz calculation behind the continuum certificate.
A final reproducibility note describes the companion certificate package.

\section{Definitions and notation}\label{sec:notation}

\begin{definition}[Checkerboard colour classes]\label{def:classes}
Let $n\in\NN$, write $[n]_0:=\{0,1,\dots,n-1\}$, and let $G_n:=[n]_0^2\subset\ZZ^2$ be the
$n\times n$ grid.  For $\eps\in\{0,1\}$ define the parity (checkerboard) classes
\[
C_\eps:=\{(x,y)\in G_n:\ x+y\equiv \eps \pmod 2\}.
\]
We will refer to $C_0$ and $C_1$ as the two \emph{colours}.  The labels ``black'' and ``white'' are
immaterial.
\end{definition}

\begin{definition}[No-three-in-line (NTIL) and monochromatic optima]\label{def:ntil}
A finite set $S\subset\ZZ^2$ is \emph{no-three-in-line} (NTIL) if no three points of $S$ are collinear.
For fixed $n\in\NN$ and $\eps\in\{0,1\}$, define
\[
D_{\mathrm{mono}}(n,\eps):=
\max\{|S|:\ S\subseteq C_\eps \text{ is NTIL}\}.
\]
We then set
\[
D_{\mathrm{mono}}(n)=\max_{\eps\in\{0,1\}}D_{\mathrm{mono}}(n,\eps).
\]
\end{definition}

The term ``monochromatic'' in Definition~\ref{def:ntil} is therefore shorthand for
``contained in one prescribed parity class of the fixed checkerboard colouring''.

For $c\in\ZZ$ we write
\[
D^{+}_c:=\{(x,y)\in G_n:\ x-y=c\},\qquad
D^{-}_c:=\{(x,y)\in G_n:\ x+y=c\},
\]
for the diagonals of slope $+1$ and $-1$ respectively (intersected with the grid).

\begin{remark}[Slope-$\pm1$ diagonals are monochromatic]\label{lem:diag-mono}
Each diagonal $D_c^+$ and $D_c^-$ lies entirely in a single parity class, since along $D_c^+$ one
has $x+y\equiv 2y+c\equiv c\pmod 2$, while along $D_c^-$ one has $x+y=c$ identically.
\end{remark}

\begin{remark}[Fat and thin classes]\label{rem:fat-thin}
For odd $n=2m+1$, the class $C_0$ has $2m^2+2m+1=(n^2+1)/2$ points and the class $C_1$ has
$2m^2+2m=(n^2-1)/2$ points.  We call $C_0$ the \emph{fat} class and $C_1$ the \emph{thin} class.
For even $n$ the two colour classes have the same cardinality, so this terminology is unnecessary.
\end{remark}

\section{Elementary diagonal bound}\label{sec:capacity}

The first obstruction comes from Remark~\ref{lem:diag-mono}.  Inside a fixed colour class, every diagonal of slope $+1$ and every diagonal of slope $-1$ is monochromatic.

\begin{proposition}[Diagonal capacity]\label{prop:capacity-total}
Fix $n\ge 2$ and $\eps\in\{0,1\}$.  The parity class $C_\eps$ has total capacity $2n-2$ on the
diagonals of slope $+1$.  Consequently every monochromatic NTIL set $S\subseteq C_\eps$ satisfies
\[
|S|\le 2n-2.
\]
\end{proposition}

\begin{proof}
Along a diagonal of slope $+1$, the parity of $x+y$ is constant.  If $C_\eps$ meets $n$ such diagonals,
then two are singleton diagonals and the remaining $n-2$ have NTIL capacity two, giving total capacity
$1+1+2(n-2)=2n-2$.  If $C_\eps$ meets $n-1$ such diagonals, each has capacity two, again giving
$2n-2$.
\end{proof}

\begin{corollary}[No completion of dilated configurations]\label{cor:no-dilation}
Let $n\ge2$.  No $2n$-point no-three-in-line set on the $n\times n$ grid, dilated by a factor of
two into the $2n\times2n$ grid, can be completed to a $4n$-point no-three-in-line set.
\end{corollary}

\begin{proof}
As explained in the introduction, all dilated points have $x+y$ even, and all added points would
have to lie in odd rows and odd columns, so they too would have $x+y$ even.  A completion would
therefore be a monochromatic no-three-in-line set of size $4n$ on the $2n\times2n$ grid, whereas
Proposition~\ref{prop:capacity-total} gives $D_{\mathrm{mono}}(2n)\le 2(2n)-2=4n-2<4n$.
\end{proof}

\begin{remark}[The sharper bound $2n-4$]\label{rem:forcing-bound}
The bound of Proposition~\ref{prop:capacity-total} is not sharp for $n\ge6$.  The companion paper
\cite{aujla2026} proves, by an argument based on the first and second moments of the four line
families, that
\[
D_{\mathrm{mono}}(n)\le 2n-4 \qquad (n\ge6),
\]
and in fact the same bound for the larger class of sets having at most two points on every row,
column and slope-$\pm1$ diagonal.  The first arXiv version of the present paper suggested this
bound by a boundary-forcing heuristic, which is superseded by \cite{aujla2026} and omitted here.
The LP and continuum arguments below are independent of it.
\end{remark}

\begin{remark}[The exceptional case $n=5$]\label{rem:n5}
The case $n=5$ is exceptional.  The monochromatic set
\[
\{(0,1),(0,3),(1,0),(1,4),(3,0),(3,4),(4,1),(4,3)\}
\]
has $8=2n-2$ points and is NTIL.  Thus the threshold $n\ge6$ in Remark~\ref{rem:forcing-bound}
is best possible.
\end{remark}

\section{A four-direction linear programming bound}\label{sec:lp}

We next replace the exact monochromatic optimisation problem by a four-direction fractional relaxation.
In addition to the diagonal families $D_c^+$ and $D_c^-$, let
\[
H_y:=\{(x',y')\in G_n:\ y'=y\},
\qquad
V_x:=\{(x',y')\in G_n:\ x'=x\},
\]
denote the horizontal and vertical grid lines.

The usual exact integer-programming formulation for NTIL or general-position subset selection may
include one constraint for every grid line containing at least three available points, or equivalently
for every collinear triple \cite{froese2017,prellberg2026}.  The relaxation below is deliberately
weaker: it keeps only the four line families $0$, $\infty$, and $\pm1$.  Its usefulness comes from the
checkerboard restriction and from the explicit dual certificates, not from any general claim that four
directions control the full no-three-in-line problem.

Every monochromatic NTIL set satisfies all of the line-capacity constraints below, so these four line
families still produce a genuine relaxation of the combinatorial problem.  Instead of choosing points integrally, we assign a nonnegative
variable $z_p$ to each lattice point $p\in C_\eps$.  One should think of $z_p$ as a fractional amount of
mass placed at $p$.  We then enforce the capacity-two constraints on rows, columns, and diagonals of
slope $\pm1$.  For fixed $n$ and $\eps$, let
$L_{\mathrm{mono}}(n,\eps)$ be the optimum of
\[
\max \sum_{p\in C_\eps} z_p
\]
subject to
\[
\sum_{p\in H_y\cap C_\eps} z_p \le 2
\qquad (y\in [n]_0),
\]
\[
\sum_{p\in V_x\cap C_\eps} z_p \le 2
\qquad (x\in [n]_0),
\]
\[
\sum_{p\in D_c^+\cap C_\eps} z_p \le 2
\qquad (c\in \ZZ),
\]
\[
\sum_{p\in D_c^-\cap C_\eps} z_p \le 2
\qquad (c\in \ZZ),
\]
and $z_p\ge 0$ for all $p\in C_\eps$.

\begin{proposition}\label{prop:lp-four-direction}
For every $n\ge 2$ and $\eps\in\{0,1\}$,
\[
D_{\mathrm{mono}}(n,\eps)\le L_{\mathrm{mono}}(n,\eps).
\]
Moreover, by linear programming duality,
\[
L_{\mathrm{mono}}(n,\eps)
=
\min
2\Biggl(
\sum_{y=0}^{n-1} h_y
\;+\;
\sum_{x=0}^{n-1} v_x
\;+\;
\sum_{c=-(n-1)}^{n-1} \lambda_c^+
\;+\;
\sum_{c=0}^{2n-2} \lambda_c^-
\Biggr),
\]
where the minimum is taken over nonnegative weights
\[
h_y,\ v_x\ge 0,\qquad
\lambda_c^+\ge 0\ \ (-(n-1)\le c\le n-1),\qquad
\lambda_c^-\ge 0\ \ (0\le c\le 2n-2)
\]
satisfying
\[
h_y+v_x+\lambda_{x-y}^+ + \lambda_{x+y}^- \ge 1
\qquad\text{for every }(x,y)\in C_\eps.
\]
\end{proposition}

\begin{proof}
If $S\subseteq C_\eps$ is NTIL, then its characteristic vector is feasible for the primal linear program,
because every horizontal line, vertical line, and diagonal of slope $\pm 1$ contains at most two points
of $S$. Hence $|S|\le L_{\mathrm{mono}}(n,\eps)$, and maximising over all such sets gives
$D_{\mathrm{mono}}(n,\eps)\le L_{\mathrm{mono}}(n,\eps)$.
The displayed dual formulation is the standard dual of the primal packing problem; see, for example,
Schrijver~\cite{schrijver1998}.
The relevant diagonal offsets are $-(n-1),\dots,n-1$ for $D_c^+$ and $0,\dots,2n-2$ for $D_c^-$,
because all other diagonals miss the grid.  Thus the dual is a finite linear program.
\end{proof}
The dual variables have a simple interpretation.  The weights $h_y$, $v_x$, $\lambda_c^+$, and
$\lambda_c^-$ are assigned to rows, columns, and diagonals, and the constraint
\[
h_y+v_x+\lambda_{x-y}^+ + \lambda_{x+y}^- \ge 1
\]
means that every admissible lattice point $(x,y)\in C_\eps$ is covered with total weight at least $1$.
Any feasible choice of these dual weights is therefore an \emph{upper-bound certificate}, because its
objective value is automatically an upper bound for $L_{\mathrm{mono}}(n,\eps)$, and hence for
$D_{\mathrm{mono}}(n,\eps)$ as well.

This line-cover language is close to, but distinct from, LP/ILP viewpoints in broader
General Position Subset Selection work.  Cao formulated a related line-cover dual problem in his
thesis, and Froese, Kanj, Nichterlein and Niedermeier discuss an ILP formulation for General
Position Subset Selection that is dual to an ILP for Point Line Cover \cite{cao2012,froese2017}.
The exact rows/columns/slope-$\pm1$ checkerboard relaxation and the continuum certificate below are
specific to the present problem; the novelty claimed here is not LP duality in the abstract.

The quantity $L_{\mathrm{mono}}(n,\eps)$ ignores all slopes other than $0$, $\infty$ and $\pm 1$.
For the classical square-grid no-three-in-line problem, the analogous four-direction relaxation collapses
to the trivial upper bound $2n$.  Indeed, the constant fractional assignment $z_{x,y}=2/n$ is primal
feasible, while the dual solution assigning weight $1/2$ to each row and each column, and weight $0$ to
the two diagonal families, has the same value.
For the checkerboard problem, however, the same relaxation is substantially more informative in the
computed range.  The exact values $D_{\mathrm{mono}}(n,\eps)$ reported below were obtained by direct exhaustive
search for $2\le n\le 16$, while the LP values were obtained by solving the corresponding primal
linear programs.  These finite computations are evidence only; they are not used in the proof of the
continuum certificate.

\begin{table}[ht]
\centering
\small
\begin{tabular}{ccccc}
\toprule
$n$ & $L_{\mathrm{mono}}(n,0)$ & $D_{\mathrm{mono}}(n,0)$ & $L_{\mathrm{mono}}(n,1)$ & $D_{\mathrm{mono}}(n,1)$\\
\midrule
2  & 2.000 & 2  & 2.000 & 2 \\
3  & 4.000 & 4  & 4.000 & 4 \\
4  & 6.000 & 6  & 6.000 & 6 \\
5  & 7.200 & 7  & 8.000 & 8 \\
6  & 9.000 & 8  & 9.000 & 8 \\
7  & 10.667 & 10 & 10.667 & 10 \\
8  & 12.333 & 12 & 12.333 & 12 \\
9  & 14.133 & 14 & 13.714 & 13 \\
10 & 15.556 & 15 & 15.556 & 15 \\
11 & 17.120 & 16 & 17.000 & 16 \\
12 & 18.750 & 18 & 18.750 & 18 \\
13 & 20.267 & 20 & 20.364 & 20 \\
14 & 22.000 & 21 & 22.000 & 21 \\
15 & 23.478 & 23 & 23.500 & 23 \\
16 & 25.091 & 24 & 25.091 & 24 \\
\bottomrule
\end{tabular}
\caption{Comparison of the four-direction LP bound $L_{\mathrm{mono}}(n,\eps)$ with the exact checkerboard NTIL maxima $D_{\mathrm{mono}}(n,\eps)$ for $2\le n\le 16$.}
\label{tab:lp-smalln}
\end{table}

In the computed range $2\le n\le 16$, using the exact LP optima rather than the rounded values displayed in
Table~\ref{tab:lp-smalln}, we find
\[
\lfloor L_{\mathrm{mono}}(n,\eps)\rfloor = D_{\mathrm{mono}}(n,\eps)
\]
for all cases other than $n\in\{6,11,14,16\}$.  At those four side lengths, both parity classes satisfy
\[
\lfloor L_{\mathrm{mono}}(n,\eps)\rfloor = D_{\mathrm{mono}}(n,\eps)+1.
\]
Thus, in the computed range, the four-direction LP is close to the exact checkerboard values.  The
remaining LP gap shows that these four line families alone do not determine the exact NTIL optimum
in every case.

\subsection{Three symmetry-reduced dual LPs}\label{subsec:lp-reduced-cases}

The generic dual in Proposition~\ref{prop:lp-four-direction} hides an important distinction.
There are three symmetry types.  For odd side length the full dihedral symmetry preserves each
colour class, whereas for even side length a $90^\circ$ rotation swaps the two colours and therefore
is unavailable in a fixed one-colour optimisation problem.  The following proposition records the
resulting reduced dual programs.  The variables in these programs are orbit-averaged versions of the
row, column, and diagonal weights in Proposition~\ref{prop:lp-four-direction}.

The reduction is an averaging argument in the dual.  Since the objective and constraints are invariant
under the colour-preserving symmetries, an optimal dual weighting may be chosen constant on line
orbits.  The work is to list these orbits and their objective coefficients in the three parity cases.

\begin{proposition}[Symmetry-reduced dual programs]\label{prop:symmetry-reduced-duals}
The four-direction dual LP may be reduced as follows.

\begin{enumerate}
\item \textbf{Odd side length, fat class.}
Let $n=2m+1$ and consider the fat class $C_0$.  Then
\[
L_{\mathrm{mono}}(2m+1,0)=\min \Phi_m^{\mathrm{fat}}(a,b),
\]
where
\[
\Phi_m^{\mathrm{fat}}(a,b)
:=
8\sum_{i=0}^{m-1}(a_i+b_i)+4(a_m+b_m),
\]
the minimum is over nonnegative $a_0,\dots,a_m$ and $b_0,\dots,b_m$, and the constraints are
\[
a_u+a_{m-v}+b_{u+v}+b_{u-v}\ge 1
\qquad (0\le v\le u,\ u+v\le m).
\]

\item \textbf{Odd side length, thin class.}
Let $n=2m+1$ and consider the thin class $C_1$.  Then
\[
L_{\mathrm{mono}}(2m+1,1)=\min \Phi_m^{\mathrm{thin}}(a,b),
\]
where
\[
\Phi_m^{\mathrm{thin}}(a,b)
:=
8\sum_{i=0}^{m-1}a_i+8\sum_{i=0}^{m-1}b_i+4b_m,
\]
the minimum is over nonnegative $a_0,\dots,a_{m-1}$ and $b_0,\dots,b_m$, and the constraints are
\[
a_u+a_{m-v-1}+b_{u+v+1}+b_{u-v}\ge 1
\qquad (0\le v\le u,\ u+v\le m-1).
\]

\item \textbf{Even side length.}
Let $n=2m$ and fix either colour class, say $C_0$.  Then
\[
L_{\mathrm{mono}}(2m,0)=\min \Phi_m^{\mathrm{even}}(a,b,c),
\]
where
\[
\Phi_m^{\mathrm{even}}(a,b,c)
:=
8\sum_{i=0}^{m-1}a_i+4\sum_{i=0}^{m-1}b_i+2c_0+4\sum_{j=1}^{m-1}c_j,
\]
the minimum is over nonnegative $a_0,\dots,a_{m-1}$, $b_0,\dots,b_{m-1}$ and
$c_0,\dots,c_{m-1}$, and the constraints are
\[
a_{u-v}+a_{\min(u+v,\,2m-1-u-v)}+b_u+c_v\ge 1
\qquad (0\le v\le u\le m-1).
\]
The other parity class has the same optimum by reflection.
\end{enumerate}
\end{proposition}

\begin{proof}
The proof is just orbit averaging in the dual.  The objective and constraints of the generic dual in
Proposition~\ref{prop:lp-four-direction} are invariant under every symmetry of the square that
preserves the chosen colour class and the four permitted line families.  Averaging a feasible dual
weighting over this symmetry group keeps it feasible and leaves the objective unchanged.  Hence an
optimum can be chosen constant on line orbits.  Conversely, any feasible reduced weighting defines a
feasible weighting for the original dual by assigning the displayed value to every line in the
corresponding orbit.  Thus the reduced and unreduced optima agree.

For the odd-fat case, the full dihedral group $D_4$ preserves $C_0$.  Using the folded coordinates
\[
(x,y)=(u+v,u-v),\qquad 0\le v\le u,\qquad u+v\le m,
\]
every point of the fat class has a representative in this triangle.  The two slope-$\pm1$ diagonal
families form one orbit type, represented by $a_0,\dots,a_m$, and the row/column families form a
second, represented by $b_0,\dots,b_m$.  The point represented by $(u,v)$ is covered by the four orbit
weights
\[
a_u,\qquad a_{m-v},\qquad b_{u+v},\qquad b_{u-v},
\]
which gives the displayed constraint.  The coefficient of each variable in the reduced objective is
twice the size of the corresponding line orbit, because the generic dual objective has the prefactor
$2$.  Thus the noncentral orbits have coefficient $8$ and the central orbits have coefficient $4$.
For example, when $n=7$ and $(u,v)=(2,1)$, the representative point is $(x,y)=(3,1)$, and the
constraint reads $a_2+a_2+b_3+b_1\ge1$.

For the odd-thin class the same averaging group is available, but the fundamental representatives are
shifted by half a layer:
\[
(x,y)=(u+v+1,u-v),\qquad 0\le v\le u,
\qquad u+v\le m-1.
\]
This accounts for the indices $m-v-1$ and $u+v+1$ in the constraint.  There is no central diagonal
orbit of the same type as in the fat case, so the $a$-sequence has length $m$, while the row/column
sequence still contains the central term $b_m$ with orbit coefficient $4$.

For even side length, a quarter-turn swaps the two colour classes, so the averaging group is smaller.
In the coordinates
\[
u=\frac{x+y}{2},\qquad v=\frac{y-x}{2},
\]
a fixed colour class becomes the integer diamond
\[
0\le u\le 2m-1,
\qquad
|v|\le \min(u,2m-1-u).
\]
After folding by the residual colour-preserving symmetries, one may take
$0\le v\le u\le m-1$.  The two oblique line families remain in one orbit type, represented by
$a_0,\dots,a_{m-1}$.  The $u$- and $v$-families remain distinct and are represented by
$b_0,\dots,b_{m-1}$ and $c_0,\dots,c_{m-1}$.  This gives the constraint
\[
a_{u-v}+a_{\min(u+v,\,2m-1-u-v)}+b_u+c_v\ge1.
\]
Again the objective coefficients are twice the corresponding line-orbit sizes, giving coefficients
$8$, $4$, and $2$ exactly as displayed.
\end{proof}
\FloatBarrier

The proposition is a structural compression of the dual.  It has the same optimum as the original
four-direction dual.  Its value is that it exposes the one-dimensional profile structure seen in the
computed optima; this profile structure is empirical evidence, not an additional theorem.  In the
odd-fat case the finite optimal profiles show, on observed active ranges, the
interior curvature relation
\[
\Delta^2 a_i+2\,\Delta^2 b_j\approx0,
\]
suggested by second-differencing tight constraints in the active region.  Representative
second-difference data are recorded in Appendix~\ref{app:profile-curvature}.  Computations in the
thin and even cases show similar profile patterns, with the shifted indexing in the thin case and the
extra $c$-profile in the even case.

The computed reduced LPs are consistent with the same candidate limiting slope.  Table~\ref{tab:odd-fat-convergence}
shows regenerated odd-fat values at larger sizes.  These values are already very close to the continuum certificate value \(\alpha\).
Like the finite exact-search and LP values in Table~\ref{tab:lp-smalln}, these computations are used as
motivation and are not part of the companion certificate package.

\begin{table}[t]
\centering
\caption{Numerical evidence for the odd-fat reduced LP.  Here $n=2m+1$ and the comparison value is the continuum certificate value $\alpha$.}
\label{tab:odd-fat-convergence}
\small
\begin{tabular}{rrr r}
\toprule
$m$ & $n$ & $L_m^{\mathrm{fat}}/n$ & $(L_m^{\mathrm{fat}}/n)-\alpha$ \\
\midrule
 40 &  81 & $1.576420575749053$ & $-0.000402821124755$ \\
 80 & 161 & $1.576808190150374$ & $-0.000015206723434$ \\
120 & 241 & $1.576771111476259$ & $-0.000052285397549$ \\
160 & 321 & $1.576820185861457$ & $-0.000003211012351$ \\
\bottomrule
\end{tabular}
\end{table}

Complementary computations in the thin and even cases give a similar numerical picture.  For example,
at sizes around $n\approx160$ one finds
\[
\frac{L_{\mathrm{mono}}(160,0)}{160}\approx 1.576757,\qquad
\frac{L_{\mathrm{mono}}(161,0)}{161}\approx 1.576808,\qquad
\frac{L_{\mathrm{mono}}(161,1)}{161}\approx 1.576704.
\]
Thus the reduced LP computations suggest a common limiting ratio with candidate value $\alpha$.
The continuum computation in the next subsection supplies an exact upper-bound certificate at this
candidate value, and Subsection~\ref{subsec:transfer} transfers it to the finite grids, proving that
$\alpha$ is an upper bound for all three limiting ratios.  With this value of $\alpha$, the finite
computations motivate two related conjectures, which are deliberately separated by strength.  In
view of Theorem~\ref{thm:discrete-bound} below, the content of both conjectures is the matching
lower bound.

\begin{conjecture}[Four-direction LP asymptotics]\label{conj:lp-asymptotics}
Let $\alpha$ be the middle real root of
\[
401\alpha^3-1744\alpha^2+2240\alpha-768=0.
\]
For odd $n=2m+1$, let $L_m^{\mathrm{fat}}$ and $L_m^{\mathrm{thin}}$ denote the relaxation
optima on the fat and thin colour classes, respectively.  For even $n=2m$, let
$L_m^{\mathrm{even}}$ denote the relaxation optimum on either colour class.  Then the three symmetry types
have the common limiting ratio
\[
\lim_{m\to\infty}\frac{L_m^{\mathrm{fat}}}{2m+1}
=
\lim_{m\to\infty}\frac{L_m^{\mathrm{thin}}}{2m+1}
=
\lim_{m\to\infty}\frac{L_m^{\mathrm{even}}}{2m}
=
\alpha.
\]
\end{conjecture}

\begin{conjecture}[Monochromatic NTIL asymptotics]\label{conj:ntil-asymptotics}
The exact monochromatic NTIL maxima satisfy
\[
\lim_{n\to\infty}\frac{D_{\mathrm{mono}}(n)}{n}=\alpha,
\]
with the same constant $\alpha$ as in Conjecture~\ref{conj:lp-asymptotics}.
\end{conjecture}

Conjecture~\ref{conj:lp-asymptotics} is the LP asymptotic suggested by the large finite relaxations.
Conjecture~\ref{conj:ntil-asymptotics} is stronger and more speculative: it does not follow from
Conjecture~\ref{conj:lp-asymptotics}, because the four-direction LP is only an upper-bound relaxation
of the exact NTIL problem.  Exact computation of $D_{\mathrm{mono}}(n,\eps)$ is currently available for the
small range shown in Table~\ref{tab:lp-smalln}; see \cite{aujla2026} for independently checkable
certificates of these values and for exact values reported by others for $17\le n\le 21$ and
$n=23$.  In the
next subsection we prove the continuum statement that supports the constant itself, namely a
feasible odd-fat continuum dual pair of objective value $\alpha$.

\subsection{The odd-fat continuum certificate}\label{subsec:odd-fat-continuum}

We now isolate the continuum calculation suggested by the odd-fat reduced dual LP.  This subsection
gives an explicit feasible dual certificate for the formal continuum problem obtained from the scaled
odd-fat reduced dual LP and computes the certificate objective value exactly.

\subsubsection*{The continuum dual problem}

Let $n=2m+1$ and consider the fat colour class $C_0$.  Passing formally, not as a proved
continuum-to-discrete limit theorem, to the scaling $u/m\to x$, $v/m\to y$, the odd-fat reduced dual
program in Proposition~\ref{prop:symmetry-reduced-duals} leads to the following continuum dual
problem.  Define
\[
\Lambda_{\mathrm{fat}}:=\inf 4\int_0^1 \bigl(A(t)+B(t)\bigr)\,dt,
\]
where the infimum is over nonnegative functions $A,B$ on $[0,1]$ satisfying
\[
F(x,y):=A(x)+A(1-y)+B(x+y)+B(x-y)\ge 1
\]
on the triangle
\[
T:=\{(x,y):0\le y\le x,\ x+y\le 1\}.
\]
The theorem below gives an explicit feasible pair $(A,B)$ and hence the upper bound
$\Lambda_{\mathrm{fat}}\le\alpha$.

\subsubsection*{The certificate}

Let $p$ be the real root in the isolating interval
\[
\frac{2115883}{10^7}<p<\frac{2115884}{10^7}
\]
of
\begin{equation}\label{eq:p-cubic}
401p^3-331p^2+19p+7=0.
\end{equation}
Set
\begin{equation}\label{eq:continuum-breakpoints}
c=\frac{187p^3-211p^2+61p-5}{2(71p^2-66p+11)},\qquad
 d=1+p-c,
\end{equation}
\[
e=\frac{2c+3p-1}{4},\qquad
f=\frac{-2c+5p+1}{4},\qquad
g=\frac{1-3p+2c}{2}.
\]
For these breakpoints define
\begin{align*}
\mathcal C={}&-c^2+\frac12ce+\frac12cf+d^2-\frac12de-\frac12df
-2d-g^2+2p^2+1,\\
\mathcal D={}&-c-d-2g+4p,\\
\mathcal E={}&-2c^2+ce+cf-\frac12ef-\frac12e-\frac12f-\frac12g^2+2p^2,\\
\mathcal H={}&-2c-g+4p-1.
\end{align*}
Let $K$ and $r$ be the solution of
\begin{equation}\label{eq:Kr-equations}
K\mathcal C+r\mathcal D=1,
\qquad
K\mathcal E+r\mathcal H=1.
\end{equation}
The exact sign checks recorded in the certificate package give \(K>0\) and \(r<0\).
Now put
\begin{equation}\label{eq:continuum-coefficients}
 s=-\frac K2g^2-rg,
 \qquad
 \ell=-\left(\frac K2(e+f)+r\right),
 \qquad
 q=-\frac K2ef-\frac K2g^2-rg,
\end{equation}
\[
n_1=Kp^2+2rp,
\]
\[
\nu=\frac{-2Kc^2+Kce+Kcf+2Kp^2-2cr+4pr}{2},
\]
and
\[
\begin{aligned}
n_2={}&\frac{1}{2}\bigl(-2Kc^2+Kce+Kcf+2Kd^2-Kde-Kdf-4Kd\\
&\qquad\qquad{}+2Kp^2-2cr-2dr+4pr\bigr).
\end{aligned}
\]
Define the elementary pieces
\begin{equation}\label{eq:continuum-pieces}
B_Q(t)=\frac K2t^2+rt+s,
\qquad
B_L(t)=-\ell t+q,
\end{equation}
\[
A_1(t)=-Kt^2-2rt+n_1,
\qquad
A_L(t)=\ell t+\nu,
\qquad
A_2(t)=-Kt^2+2Kt+n_2,
\]
and finally the two functions
\begin{equation}\label{eq:continuum-functions}
A(t)=
\begin{cases}
0, & 0\le t\le p,\\
A_1(t), & p\le t\le c,\\
A_L(t), & c\le t\le d,\\
A_2(t), & d\le t\le 1,
\end{cases}
\end{equation}
\[
B(t)=
\begin{cases}
B_Q(t), & 0\le t\le e,\\
B_L(t), & e\le t\le f,\\
B_Q(t), & f\le t\le g,\\
0, & g\le t\le 1.
\end{cases}
\]
The two quadratic pieces of $B$ are thus parts of the same parabola.  Figure~\ref{fig:continuum-profiles}
shows the resulting dual profiles on the unit interval.

\begin{figure}[tb]
\centering
\includegraphics[width=0.74\linewidth]{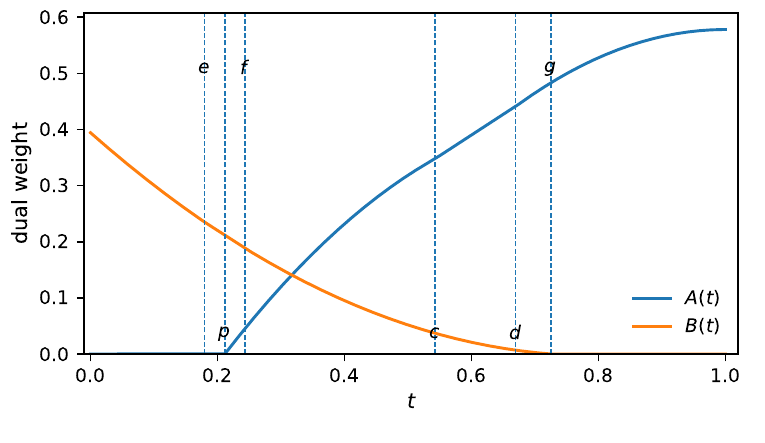}
\caption{The continuum dual profiles $A$ and $B$ on $[0,1]$.  The zero intervals and the breakpoints
$p,c,d,e,f,g$ show the piecewise-profile pattern suggested by the finite dual profiles.  The function
$A$ is zero up to $p$, while $B$ vanishes after $g$; between the breakpoints the pieces are quadratic
or linear as in \eqref{eq:continuum-functions}.}
\label{fig:continuum-profiles}
\end{figure}

The verification of the certificate has three parts.  First, endpoint and sign inequalities in the
cubic field show that the pieces define nonnegative functions $A$ and $B$ on $[0,1]$;
Figure~\ref{fig:continuum-profiles} is a visualisation of this certified pair.  Second, the obstacle slack
\[
G(x,y)=A(x)+A(1-y)+B(x+y)+B(x-y)-1
\]
is checked on the subdivision shown in Figure~\ref{fig:continuum-domain}.  On each cell it is
quadratic, so nonnegativity follows from the Bernstein coefficient certificate.  Finally, direct
integration gives the objective value $\alpha$.

\begin{theorem}[Odd-fat continuum dual certificate]\label{thm:odd-fat-continuum-certificate}
The functions $A$ and $B$ defined in \eqref{eq:continuum-functions} are feasible for the continuum dual
problem.  Consequently
\[
\Lambda_{\mathrm{fat}}\le \alpha,
\]
where $\alpha$ is the middle real root of
\[
401\alpha^3-1744\alpha^2+2240\alpha-768=0.
\]
\end{theorem}

The formulas above are the certificate used in the proof.  Its content is the feasible-pair upper bound
for the continuum problem; the discrete asymptotic conjectures remain separate.
Appendix~\ref{app:certificate-derivation} gives the active-patch and stationarity calculation leading to
these formulae.  The feasibility proof below uses only the displayed definitions and the exact algebraic
sign checks.

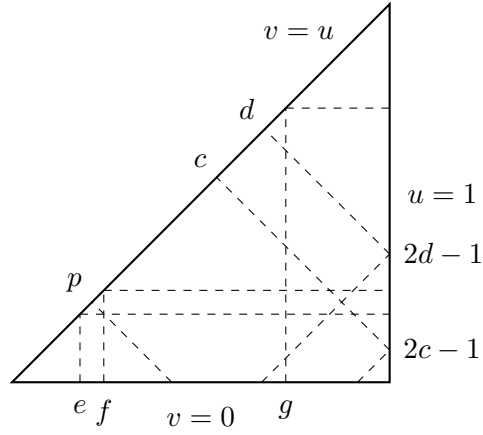
\begin{figure}[tb]
\centering
\begin{tikzpicture}[scale=5.0]
  \coordinate (O) at (0,0);
  \coordinate (A) at (1,0);
  \coordinate (B) at (1,1);
  \draw[thick] (O) -- (A) -- (B) -- cycle;

  \draw[dashed] (0.180,0) -- (0.180,0.180);
  \draw[dashed] (0.243,0) -- (0.243,0.243);
  \draw[dashed] (0.725,0) -- (0.725,0.725);
  \draw[dashed] (0.180,0.180) -- (1,0.180);
  \draw[dashed] (0.243,0.243) -- (1,0.243);
  \draw[dashed] (0.725,0.725) -- (1,0.725);

  \draw[dashed] (0.423,0) -- (0.212,0.212);
  \draw[dashed] (1,0.085) -- (0.542,0.542);
  \draw[dashed] (1,0.339) -- (0.669,0.669);

  \draw[dashed] (0.661,0) -- (1,0.339);
  \draw[dashed] (0.915,0) -- (1,0.085);

  \node[below] at (0.5,-0.04) {$v=0$};
  \node[right] at (1.02,0.5) {$u=1$};
  \node[above left] at (0.88,0.88) {$v=u$};

  \node[below] at (0.180,-0.02) {$e$};
  \node[below] at (0.243,-0.02) {$f$};
  \node[below] at (0.725,-0.02) {$g$};
  \node[above left] at (0.212,0.212) {$p$};
  \node[above left] at (0.542,0.542) {$c$};
  \node[above left] at (0.669,0.669) {$d$};
  \node[right] at (1.01,0.085) {$2c-1$};
  \node[right] at (1.01,0.339) {$2d-1$};
\end{tikzpicture}
\caption{Subdivision of the continuum triangle after the change of variables $u=x+y$ and $v=x-y$, so that $T$ becomes $0\le v\le u\le1$.  The dashed lines show the breakpoint families $u=e,f,g$, $v=e,f,g$, $u+v=2p,2c,2d$, and $u-v=2(1-d),2(1-c)$ used in the Bernstein-basis obstacle check.  The drawing is to scale using decimal approximations of the exact breakpoints.}
\label{fig:continuum-domain}
\end{figure}

\subsubsection*{Proof of the certificate}

\begin{proof}[Proof of Theorem~\ref{thm:odd-fat-continuum-certificate}]
First, the displayed breakpoints satisfy
\[
0<p<c<d<1,
\qquad
0<e<f<g<1.
\]
Moreover $K>0$, $r<0$, and the exact sign checks in the cubic field $\QQ(p)$ give
\[
A_1(c)>0,
\qquad
A_L(d)>0,
\qquad
A_2(1)>0,
\]
\[
B_Q(0)>0,
\qquad
B_Q(e)=B_L(e)>0,
\qquad
B_L(f)=B_Q(f)>0.
\]
Since $A_1$ and $A_2$ are concave on their intervals, $A_L$ and $B_L$ are linear, and $B_Q$ is
convex with zeros at $g$ and $-2r/K-g>g$, these endpoint inequalities prove
$A(t),B(t)\ge0$ on $[0,1]$.

It remains to prove the obstacle inequality
\[
G(x,y):=F(x,y)-1\ge0
\]
on $T$.  Put
\[
u=x+y,
\qquad
v=x-y.
\]
Then
\[
x=\frac{u+v}{2},
\qquad
1-y=\frac{2-u+v}{2},
\]
and the triangle $T$ becomes
\[
0\le v\le u\le 1.
\]
The breakpoints of the four arguments of $A$ and $B$ cut this triangle by the lines
\[
u=e,f,g,
\qquad
v=e,f,g,
\]
\[
u+v=2p,2c,2d,
\qquad
u-v=2(1-d),2(1-c),
\]
together with the boundary lines $v=0$, $v=u$ and $u=1$.  The line $u-v=2(1-p)$ lies outside the
triangle and plays no role.  These lines produce 24 nonempty polygonal cells.  On each cell, $G$ is a
polynomial of degree at most two in $(u,v)$.

Triangulate each polygonal cell by a fan from one vertex.  This gives 40 triangles.  On a triangle with
barycentric coordinates $(\lambda_0,\lambda_1,\lambda_2)$, write the quadratic polynomial $G$ in the
degree-two Bernstein basis,
\[
G=\sum_{i+j+k=2} b_{ijk}
\frac{2!}{i!j!k!}\lambda_0^i\lambda_1^j\lambda_2^k .
\]
If all Bernstein coefficients $b_{ijk}$ are nonnegative, then $G$ is nonnegative on the whole triangle
\cite{farouki2012}.  For a quadratic these coefficients are obtained from the three vertex values and
the three edge-midpoint values.
\[
b_{200}=G(V_0),
\qquad
b_{020}=G(V_1),
\qquad
b_{002}=G(V_2),
\]
and, for example,
\[
b_{110}=2G\left(\frac{V_0+V_1}{2}\right)
-\frac{G(V_0)+G(V_1)}{2},
\]
with the analogous formulae for $b_{101}$ and $b_{011}$.

All vertices of the 40 triangles have coordinates in $\QQ(p)$, where $p$ is the root specified
in \eqref{eq:p-cubic}.  Hence all Bernstein coefficients lie in the same cubic field.  Reducing the
240 coefficients modulo the cubic gives 102 identically zero coefficients and 138 nonzero coefficients.
The file \texttt{bernstein\_coefficients.tsv} in the companion certificate package records each coefficient as
\[
c_0+c_1p+c_2p^2\qquad(c_i\in\QQ).
\]
A Sturm-sequence computation~\cite{basu_pollack_roy2006} isolates the required root in
\[
\frac{2115883}{10^7}<p<\frac{2115884}{10^7}.
\]
On this interval, exact rational interval evaluation shows that all 138 nonzero coefficients are positive.
Thus every Bernstein coefficient is nonnegative on every triangle in the decomposition, and therefore
$G\ge0$ on $T$.

Direct integration of the functions in \eqref{eq:continuum-functions}, followed by reduction in
$\QQ(p)$, gives the objective value $4\int_0^1(A(t)+B(t))\,dt=\alpha$; the algebraic derivation of this
identity is recorded in Appendix~\ref{app:certificate-derivation}.  Therefore
$\Lambda_{\mathrm{fat}}\le\alpha$, as claimed.
\end{proof}

Theorem~\ref{thm:odd-fat-continuum-certificate} proves the continuum upper bound
$\Lambda_{\mathrm{fat}}\le\alpha$.  The next subsection transfers it to the finite grids.

\subsection{From the continuum to the grid}\label{subsec:transfer}

The continuum problem was obtained from the odd-fat reduced dual by a formal scaling.  The
following theorem, proved in the companion paper \cite{aujla2026}, shows that the passage can be
reversed: a feasible continuum pair of bounded variation can be sampled to produce feasible reduced
dual weights for every finite grid, of either parity and for either colour class, at a cost of
$O(1)$ in the objective.

\begin{theorem}[Continuum-to-discrete transference; Aujla, Prellberg and Sandhu
{\cite[Theorem~5.3]{aujla2026}}]\label{thm:transference}
Let $A,B\colon[0,1]\to[0,\infty)$ be functions of bounded variation satisfying
\[
A(x)+A(1-y)+B(x+y)+B(x-y)\ge 1
\qquad\text{for all }(x,y)\in T,
\]
and put $\beta:=4\int_0^1(A(t)+B(t))\,dt$.  Then, for every $n\ge2$ and each $\eps\in\{0,1\}$,
\[
L_{\mathrm{mono}}(n,\eps)\le \beta n+C_{A,B},
\qquad
C_{A,B}=8\bigl(\operatorname{Var}A+\operatorname{Var}B\bigr)+4\bigl(A(1)+B(1)\bigr),
\]
where $\operatorname{Var}$ denotes the total variation on $[0,1]$.
\end{theorem}

Combining Theorem~\ref{thm:transference} with the certificate of
Theorem~\ref{thm:odd-fat-continuum-certificate} gives the main result of this paper.

\begin{theorem}[Asymptotic upper bound]\label{thm:discrete-bound}
Let $\alpha\approx1.5768$ be the middle real root of $401\alpha^3-1744\alpha^2+2240\alpha-768$.
There is an absolute constant $C$ such that, for every $n\ge2$ and each $\eps\in\{0,1\}$,
\[
D_{\mathrm{mono}}(n,\eps)\le L_{\mathrm{mono}}(n,\eps)\le\alpha n+C.
\]
In particular,
\[
\limsup_{n\to\infty}\frac{L_{\mathrm{mono}}(n,\eps)}{n}\le\alpha
\qquad\text{and}\qquad
\limsup_{n\to\infty}\frac{D_{\mathrm{mono}}(n)}{n}\le\alpha .
\]
\end{theorem}

\begin{proof}
The functions $A$ and $B$ of Theorem~\ref{thm:odd-fat-continuum-certificate} are nonnegative and
continuous, and piecewise polynomial on $[0,1]$; in particular they have bounded variation.  They
satisfy the obstacle inequality on the closed triangle $T$, boundary included, because the
Bernstein-basis certificate is evaluated on closed triangles.  Their objective value is
$4\int_0^1(A+B)=\alpha$.  Theorem~\ref{thm:transference} applies with $\beta=\alpha$ and
$C=C_{A,B}$, and $D_{\mathrm{mono}}(n,\eps)\le L_{\mathrm{mono}}(n,\eps)$ is
Proposition~\ref{prop:lp-four-direction}.
\end{proof}

We sketch the proof of Theorem~\ref{thm:transference} in the odd-fat case, which contains the
main idea; the argument is due to Sandhu \cite{aujla2026} and was also suggested independently by
a referee of this paper.  Let $n=2m+1$ and consider the fat class.  Set
\[
a_u:=A(u/m),\qquad b_j:=B(j/m)\qquad(0\le u,j\le m).
\]
For indices $0\le v\le u$ with $u+v\le m$, the point $(x,y)=(u/m,v/m)$ lies in $T$, and the reduced
constraint of Proposition~\ref{prop:symmetry-reduced-duals}(1),
\[
a_u+a_{m-v}+b_{u+v}+b_{u-v}\ge1,
\]
is exactly the obstacle inequality $A(x)+A(1-y)+B(x+y)+B(x-y)\ge1$ at this point.  Hence the
sampled weights are feasible, and
\[
L_{\mathrm{mono}}(2m+1,0)\le\Phi_m^{\mathrm{fat}}(a,b)
=8\sum_{i=0}^{m-1}(A+B)(i/m)+4\bigl(A(1)+B(1)\bigr).
\]
The sum is a left Riemann sum, and for a function $f$ of bounded variation on $[0,1]$ one has
$\bigl|\sum_{i=0}^{m-1}f(i/m)-m\int_0^1f\bigr|\le\operatorname{Var}f$.  Therefore
\[
\Phi_m^{\mathrm{fat}}(a,b)\le 8m\int_0^1(A+B)+C_{A,B}=2m\beta+C_{A,B}\le\beta n+C_{A,B}.
\]
For the thin class the $a$-samples are taken at the half-mesh points $(u+\tfrac12)/m$, which
matches the shifted indices in Proposition~\ref{prop:symmetry-reduced-duals}(2), and the even case
follows from the monotonicity $L_{\mathrm{mono}}(n,\eps)\le L_{\mathrm{mono}}(n+1,\eps)$, obtained
by extending a feasible primal mass assignment by zero.  The details are in
\cite[Section~5]{aujla2026}.

Theorem~\ref{thm:discrete-bound} is one-sided.  It proves neither that the continuum certificate
is optimal nor that the normalised finite LP optima converge; Conjectures~\ref{conj:lp-asymptotics}
and~\ref{conj:ntil-asymptotics} assert the matching lower bounds.  Relating the LP asymptotic to
the exact monochromatic NTIL optimum remains a separate problem.

\appendix

\section{Finite odd-fat profile curvature data}\label{app:profile-curvature}

This appendix records the finite-profile data referred to in Subsection~\ref{subsec:lp-reduced-cases}.
The data are not used in the proof of Theorem~\ref{thm:odd-fat-continuum-certificate}.  They document one
observed interior curvature pattern that motivated the piecewise quadratic-linear certificate.

\begin{table}[h]
\centering
\caption{Second-difference evidence for the odd-fat reduced dual profiles.  The windows are interior
active ranges of the finite optimal profiles.  The displayed averages are scaled by $m^2$.}
\label{tab:odd-fat-curvature}
\small
\begin{tabular}{rccrc r}
\toprule
$m$ & $a$ window & $m^2\operatorname{avg}\Delta^2 a$ & $b$ window & $m^2\operatorname{avg}\Delta^2 b$ & $\Delta^2 a/\Delta^2 b$ \\
\midrule
 40 & $30\ldots38$   & $-2.4794669146$ & $15\ldots22$ & $ 1.2397334573$ & $-2.0000000000$ \\
 80 & $60\ldots78$   & $-2.5053826581$ & $30\ldots44$ & $ 1.2526913290$ & $-2.0000000000$ \\
120 & $90\ldots118$  & $-2.4811188008$ & $45\ldots66$ & $ 1.2405594004$ & $-2.0000000000$ \\
160 & $120\ldots158$ & $-2.4686298381$ & $60\ldots88$ & $ 1.2343149191$ & $-2.0000000000$ \\
\bottomrule
\end{tabular}
\end{table}

\section{Derivation of the odd-fat continuum certificate}\label{app:certificate-derivation}

This appendix records the symbolic ansatz calculation that produced the breakpoints and coefficients
used in the continuum certificate of Subsection~\ref{subsec:odd-fat-continuum}.  The proof of
Theorem~\ref{thm:odd-fat-continuum-certificate} uses the resulting functions directly; the calculation
below explains how the active pieces were found.

The finite odd-fat LP profiles motivate the piecewise structure displayed in
\eqref{eq:continuum-functions}.  The function $A$ is zero near the lower endpoint, then quadratic,
linear, and quadratic again, while $B$ is quadratic, linear, quadratic, and then zero.  Continuity and the endpoint
conditions force the coefficient relations in \eqref{eq:continuum-coefficients}.  In particular,
\[
\nu=A_1(c)-\ell c,
\qquad
n_2=A_L(d)+Kd^2-2Kd.
\]
The quadratic active patch
\[
A_1(x)+A_2(1-y)+B_Q(x+y)+B_Q(x-y)=1
\]
contributes the scalar equation
\[
Q:=n_1+n_2+K+2s-1=0.
\]
The linear active patch
\[
A_L(x)+A_L(1-y)+B_L(x-y)=1
\]
contributes
\[
R:=\ell+2\nu+q-1=0.
\]
After substituting the coefficient relations, these two equations take the linear form
\[
Q=K\mathcal C+r\mathcal D-1,
\qquad
R=K\mathcal E+r\mathcal H-1,
\]
where $\mathcal C,\mathcal D,\mathcal E,\mathcal H$ are the expressions defined in the certificate.  Thus
\eqref{eq:Kr-equations} is exactly the condition $Q=R=0$.

The objective value of the piecewise pair is
\[
\begin{aligned}
\Phi=4\Bigl(&
\int_p^c A_1(t)\,dt+
\int_c^d A_L(t)\,dt+
\int_d^1 A_2(t)\,dt\\
&{}+
\int_0^e B_Q(t)\,dt+
\int_e^f B_L(t)\,dt+
\int_f^g B_Q(t)\,dt
\Bigr).
\end{aligned}
\]
Performing the integrations and substituting the coefficient relations gives
\[
\Phi=K\mathcal A+r\mathcal B,
\]
where
\[
\begin{aligned}
\mathcal A={}&
\frac83c^3-c^2e-c^2f-4c^2+2ce+2cf
-\frac83d^3+d^2e+d^2f+8d^2\\
&-2de-2df-8d-\frac13e^3+e^2f-ef^2+\frac13f^3
-\frac43g^3-\frac83p^3+4p^2+\frac83,
\end{aligned}
\]
and
\[
\mathcal B=2c^2-4c+2d^2-4d-2g^2-4p^2+8p.
\]
Introducing Lagrange multipliers $\lambda,\mu$ for the two constraints $Q=R=0$, the stationarity
equations for
\[
\mathcal L=\Phi+\lambda Q+\mu R
\]
are
\[
Q=0,
\qquad
R=0,
\]
\[
\mathcal A+\lambda\mathcal C+\mu\mathcal E=0,
\qquad
\mathcal B+\lambda\mathcal D+\mu\mathcal H=0,
\]
\[
(Kp+r)(2p-2-\lambda-\mu)=0,
\]
\[
(4c-\lambda-2\mu-4)(4Kc-Ke-Kf+2r)=0,
\]
\[
(4d-\lambda-4)(-4Kd+Ke+Kf+4K+2r)=0,
\]
\[
\begin{aligned}
0={}&2c^2-c\lambda-2c\mu-4c-2d^2+d\lambda+4d\\
&+2e^2-4ef+2f^2+f\mu+\mu,
\end{aligned}
\]
\[
\begin{aligned}
0={}&2c^2-c\lambda-2c\mu-4c-2d^2+d\lambda+4d\\
&-2e^2+4ef+e\mu-e-2f^2+\mu,
\end{aligned}
\]
and
\[
(Kg+r)(4g+2\lambda+\mu)=0.
\]
The branch used in the certificate is the nondegenerate branch
\[
\lambda+\mu=2p-2,
\qquad
\lambda+2\mu=4c-4,
\qquad
\lambda=4d-4,
\qquad
2\lambda+\mu=-4g.
\]
Equivalently,
\[
d=1+p-c,
\qquad
g=\frac{1-3p+2c}{2},
\qquad
\lambda=4(p-c),
\qquad
\mu=4c-2p-2.
\]
On this branch, adding and subtracting the two equations obtained by varying $e$ and $f$ gives
\[
(2p-e-f)(-2c+p+1)=0,
\]
and
\[
(e-f)(-2c+2e-2f+p+1)=0.
\]
The branch matching the certified breakpoint ordering is therefore
\[
e+f=2p,
\qquad
f-e=\frac{1+p-2c}{2},
\]
which gives the formulae for $e$ and $f$ used in the certificate.

It remains to solve for $p$ and $c$.  Substituting these branch relations into the stationarity
equations for $r$ and $K$ gives
\[
-4c^2-4cp+4c+13p^2-10p+1=0,
\]
and
\[
88c^3-36c^2p-36c^2-126cp^2+132cp-30c
+29p^3-57p^2+39p-3=0.
\]
Eliminating $c$ gives
\[
(p-1)^2(5p-1)(401p^3-331p^2+19p+7)=0.
\]
The ordering constraints single out the middle real root of \eqref{eq:p-cubic}.  Substitution into the
objective and exact reduction in $\QQ(p)$ gives
\[
\Phi=\alpha,
\]
where eliminating $p$ yields
\[
401\alpha^3-1744\alpha^2+2240\alpha-768=0,
\]
and the root compatible with the numerical value of the integral is the middle real root.

\section*{Computational reproducibility and certificate package}
The continuum feasibility part of Theorem~\ref{thm:odd-fat-continuum-certificate} is a finite exact
algebraic check in the cubic field $\QQ(p)$.  The companion directory \texttt{odd\_fat\_certificate/}
contains \texttt{README.md}, \texttt{constants.tsv}, \texttt{sign\_checks.tsv}, \texttt{triangles.tsv},
\texttt{bernstein\_coefficients.tsv}, and \texttt{verify\_odd\_fat\_certificate.py}.  These files record
the exact constants, the exact sign checks, the triangulation of the continuum domain, and the exact
Bernstein coefficients for the obstacle polynomial on each triangle.  The verification script regenerates
the certificate and uses exact rational arithmetic for the algebraic sign checks; it requires Python~3 and
\texttt{shapely} for the subdivision step.  The check is run from within
\texttt{odd\_fat\_certificate/} by executing \texttt{python verify\_odd\_fat\_certificate.py}.  In the regenerated certificate the 240 Bernstein coefficients
split into 102 identically zero coefficients and 138 positive coefficients, all represented as elements of
$\QQ(p)$.  The package verifies the continuum obstacle certificate; it does not certify the finite LP
values in Tables~\ref{tab:lp-smalln} and~\ref{tab:odd-fat-convergence}, the exact NTIL search data, or
the conjectures.  The exact value of the integral is the symbolic calculation
recorded in Appendix~\ref{app:certificate-derivation}.  An independent exact verification of the
continuum certificate, together with checkable certificates for the finite values
$D_{\mathrm{mono}}(n,\eps)$ in Table~\ref{tab:lp-smalln}, is provided with \cite{aujla2026}.

\section*{Acknowledgements}
The author thanks the two referees for their careful reading and helpful suggestions.  In
particular, one referee independently proposed the sampling argument sketched in
Subsection~\ref{subsec:transfer} and pointed out Corollary~\ref{cor:no-dilation}.  The author also
thanks Jujhar Aujla and Nirvair Sandhu for agreeing to the inclusion of
Theorem~\ref{thm:transference} and its proof sketch in this paper.

\section*{Declaration of generative AI and AI-assisted technologies}
During the preparation of this manuscript the author used OpenAI ChatGPT to assist with editorial
checks and \LaTeX{} organisation. The mathematical content was reviewed and edited by the author, who
takes full responsibility for the final version.

\end{document}